\documentclass[twoside]{aiml}

\usepackage{aimlmacro}

\usepackage{amsfonts,amscd,amsmath,amssymb,amsfonts,latexsym,}
\usepackage[usenames,dvipsnames]{xcolor}
\usepackage{enumerate}
\usepackage{manfnt}
\usepackage[shortlabels]{enumitem}
\usepackage{faktor}
\usepackage{hyperref}
\hypersetup{colorlinks=true,linkcolor=blue,urlcolor=blue}

\newcommand{\holds}[1]{\:\left(#1\right)}
\newcommand{\Pow}[1]{\mathcal{P}\left(#1\right)}
\newcommand{\inv}{{}^{-1}}
\newcommand{\U}{\mathcal{U}}
\newcommand{\V}{\mathcal{V}}

\newcommand{\ML}{\mathrm{ML}}
\newcommand{\PV}{\mathrm{PV}}
\renewcommand{\b}{\Box}
\renewcommand{\d}{\Diamond}

\newcommand{\MP}{\mathbf{MP}}
\newcommand{\Gen}{\mathbf{Gen}}
\newcommand{\Sub}{\mathbf{Sub}}

\newcommand{\K}{\mathbf{K}}

\newcommand{\KV}{\mathbf{K5}}
\newcommand{\DV}{\mathbf{KD5}}

\newcommand{\DIVV}{\mathbf{KD45}}
\newcommand{\KVB}{\mathbf{K5B}}

\newcommand{\SV}{\mathbf{S5}}

\newcommand{\GLIII}{\mathbf{GL.3}}
\newcommand{\GrzIII}{\mathbf{Grz.3}}
\newcommand{\C}{\mathcal{C}}

\newcommand{\F}{\mathcal{F}}
\newcommand{\G}{\mathcal{G}}
\newcommand{\Q}{\mathcal{Q}}
\DeclareMathOperator{\dom}{dom}
\DeclareMathOperator{\Fr}{Fr}
\DeclareMathOperator{\Log}{Log}
\DeclareMathOperator{\Con}{Con}
\newcommand\gen{{\uparrow}}
\newcommand\pto{\twoheadrightarrow}
\newcommand\restr{{\upharpoonright}}
\newcommand{\out}{_{\mathrm{out}}}
\newcommand{\IN}{_{\mathrm{in}}}

\DeclareMathOperator{\pr}{P}
\renewcommand{\P}{\pr}
\DeclareMathOperator{\E}{E}
\DeclareMathOperator{\var}{Var}

\newcommand{\RF}{\hat{F}}
\newcommand{\RU}{\hat{U}}
\newcommand{\RR}{\hat{R}}
\newcommand{\Rh}{\hat{h}}
\newcommand{\RX}{\hat{X}}
\newcommand{\as}{^{\mathbf{as}}}
\newcommand{\aas}{a.a.s. }

\begin{document}

\def\lastname{Sliusarev}

\begin{frontmatter}

\title{Modal logics of almost sure validities in some classes of euclidean and transitive frames}

\begin{abstract} 
Given a class \(\C\) of finite Kripke frames, we consider the uniform distribution on the frames from \(\C\) with \(n\) states. A formula is almost surely valid in \(\C\) if the probability that it is valid in a random \(\C\)-frame with \(n\) states tends to \(1\) as \(n\) tends to infinity. The formulas that are almost surely valid in \(C\) form a normal modal logic.

We find complete and sound axiomatizations for the logics of almost sure validities in the classes of finite frames defined by the logics \( \DV\), \(\DIVV\), \( \KVB\), \( \SV\), \( \GrzIII\), and \( \GLIII.\)\bigskip

\emph{MSC class:} 03B45.
\end{abstract}


\author{Vladislav Sliusarev}
\address{New Mexico State University, USA\\\and\\ Moscow Institute of Physics and Technology, Russia}
\end{frontmatter}

The study of random structures for various logical systems is a significant research field within contemporary mathematics. A large body of work in this area concerns the first-order relational language. A key result in this field is the Zero-one law for first-order logic, which states that any first-order definable property of random relational structures, such as random graphs, has an asymptotic probability of either zero or one. This law was independently proved for the Erd\H{o}s--R\'enyi model of random graph in~\cite{Glebskii1969RangeAD} and~\cite{fagin_1976}.

Exploring the behavior of random structures for logical systems that extend beyond the first-order language is a compelling research direction. The modal language, interpreted on relational structures using Kripke semantics, is a particularly important example, which has led to several noteworthy discoveries. Le Bars disproved the zero-one law for modal logic in Kripke frames~\cite{lebars}. Verbrugge~\cite{verbrugge_grz} proved the zero-one law in the finite models of of Grzegorczyk's logic and weak Grzegorczyk's logic, and later provided a valuable example of a modal logic that satisfies the zero-one law both in models and in frames, namely the provability logic \(\textbf{GL}\) \cite{verbrugge}.

A related problem of interest involves the study of sentences that are true with a probability tending to one, termed as \emph{asymptotically almost surely true}. Gaifman~\cite{Gaifman1964ConcerningMI} provided an axiomatization for almost sure truths in the Rad\'o graph, a model of a countable random graph, in the first-order relational logic.

There are several results in the study of almost sure truths in the modal languages. The logic of almost sure truths in random Kripke models coincides with Carnap’s modal logic, according to Halpern and Kapron~\cite{halpern}. The logic of almost sure truths in transitive reflexive models is also described in~\cite{halpern}. Verbrugge~\cite{verbrugge_grz}\cite{verbrugge} provided axiomatizations for almost sure truths in the models of \(\textbf{GL},\,\textbf{Grz},\) and~\(\textbf{wGrz}\), and almost sure validities in the frames of~\(\textbf{GL}\). Goranko~\cite{goranko2} found a complete and sound axiomatization for the logic of almost sure validities in a countable frame. The paper~\cite{goranko} utilizes this result to identify some of the almost sure validities in finite frames. However, the problem of complete axiomatization of almost sure validities in the class of finite frames remains open.

In this paper we discuss the logics of almost sure validities in various classes of finite frames. We generalize Goranko's construction of the random finite frame: given a class of frames~\(\C,\) we consider the uniform distribution on the labelled frames with \(n\) states that belong to~\(\C.\) Our general result states that the almost sure validities form a normal modal logic that extends~\(\Log\C.\) We achieve an important technical result of studying such logics: for a large class of logics, the almost sure validities in a random~\(\C\)-frame are also almost surely valid in a random \emph{connected}~\(\C\)-frame (Theorem~\ref{thm:logic-of-connected}). Since the connected frames typically have a simpler combinatorial structure, this connection allows us to find upper bounds on the logic of almost sure validities in~\(\C.\) We use this general theory to find finite axiomatizations for the logics of almost sure validities in the classes of frames defined by the logics \( \DV\), \(\DIVV\), \( \KVB\), \( \SV\), \( \GrzIII,\) and \( \GLIII.\)

\section{Preliminaries}\label{sec:preliminaries}

\subsection{Modal syntax and semantics}

We consider the basic modal language \(\ML\) of formulas in the alphabet that consists of a countable set \(\PV = \{p_0,\,p_1,\,\ldots\}\) of propositional variables, classical connectives \(\to,\,\bot,\) and a unary operator \(\b.\) We use the standard abbreviations of connectives, in particular, \(\d \varphi \equiv \lnot \b \lnot \varphi.\) 

A set \(L \subseteq \ML\) is a \emph{(normal modal) logic} if \(L\) contains all propositional tautologies, the normality axiom \(\b(p\to q) \to (\b p \to \b q)\) and is closed under the rules of inference:
\begin{description}
  \item[(MP)] If \(\varphi,\,\varphi \to \psi \in L\), then \(\psi \in L,\)
  \item[(Gen)] If \(\phi \in L\), then \(\b \varphi \in L,\)
  \item[(Sub)] If \(\varphi \in L\), \(p\in \PV,\,\theta\in \ML\), and the formula \(\psi\in \ML\) is obtained from \(\varphi\) by replacing all instances of \(p\) with \(\theta,\) then \(\psi \in L.\)
\end{description}

By a \emph{frame} we mean a Kripke frame \(F = (X,\,R),\,X \ne \varnothing,\,R \subseteq X\times X;\) we refer to the elements of \(X\) as \emph{states of \(F\)}. The set of states, or \emph{domain}, of \(F\) is denoted \(\dom F.\)

The notation \(F\models \varphi\), where \(F\) is a frame and \(\varphi\in\ML\) is a formula, means `\(\varphi\) is valid in \(F\)' with the standard definition (see, for example, \cite[Definition 1.28]{BdRV}). For any set \(\Gamma\) of modal formulas, \(F\models \Gamma\) means \(\forall \varphi\in \Gamma\holds{F\models \varphi}.\) The set \(\Log F\) of all formulas \(\varphi\in \ML\) that are valid in a frame \(F\) is called the \emph{logic} of \(F.\) This definition extends to classes of frames: if~\(\F\) is a class of frames, then~\(\Log \F\) is the set of all formulas that are valid in any frame~\(F\in\F\). Given a set  of formulas \(\Gamma\subseteq \ML\), let \(\Fr\Gamma\) denote the class of all frames \(F\) such that \(F \models \Gamma.\)

\subsection{Operations on frames}

For the convenience of the reader, we recall some basic notation and techniques of modal logic that we use in the article.

Let \(X\) be a set. The \emph{diagonal relation on \(X\)} is \(Id_X = \{(a,\,a)\mid a\in X\}.\) If \(R \subseteq X\times X\) is a relation, let \(R^0 = Id_X\) and \(R^{i+1} = R\circ R^i\) for any \(i \in\omega.\) The \emph{inverse relation} \(R\inv\) is defined as \(\{(a,b)\in X\mid bRa\},\) and~\(R^{-i} \equiv {(R\inv)}^i\) for any~\(n\in\omega.\)

For any \(U \subseteq X,\) \(R\out[U]\) denotes the set \(\{a \in X \mid \exists u\in U\holds{uRa}\},\) and~\(R\IN[U]\) denotes~\(\{a\in X\mid \exists u\in U\holds{aRu}\}.\) The notations~\(R\out(a),\,R\IN(a)\), where~\(a\in X,\) abbreviate~\(R\out[\{a\}]\) and~\(R\IN[\{a\}],\) respectively.

For a relation \(R \subseteq X\times X\), define the \emph{transitive closure} \(R^+=\bigcup_{i \ge 1} R^i,\) and the \emph{reflexive transitive closure}~\(R^*=Id_X \cup R^+.\)

Given \(R \subseteq X\times X\) and \(U \subseteq X,\)  the \emph{restriction of \(R\) on \(U\)} is the relation \(R\restr U = R\cap (U\times U).\)

If \(F = (X,\,R)\) is a frame and \(U \subseteq X\), the \emph{subframe of \(X\) generated by \(U\)} is the frame \(F\gen U = (R^*\out[U],\,R\restr R^*\out[U]).\) If \(a\in X\), then \(F\gen a\) is a shorthand for \(F\gen \{a\}.\) A frame~\(F\) is said to be \emph{point-generated} if~\(F = F\gen a\) for some~\(a\in \dom F.\) The generated subframe preserves the validity of modal formulas: for any~\(U \subseteq \dom F\), \(\Log{F} \subseteq \Log{F\gen U}\)~\cite[Proposition~2.6]{BdRV}.

The \emph{disjoint sum} \(\biguplus_{i\in I} F_i \) of the family of frames \(F_i = (X_i,\,R_i),\,i \in I,\) where \(I\) is a nonempty set,  is defined as \((X,\,R)\) where 
\[
  X = \{(a,i)\mid i\in I,\,a\in X_i\};\qquad
  (a,i) R (b,j)\iff i = j\text{ and }a R_i b.
\]
 The notation \(F \uplus G\) is a shorthand for \(\biguplus_{i \in \{1,\,2\}} F_i\) where \(F_1 = F,\,F_2 = G.\) It is well-known that \(\Log \biguplus_{i\in I} F_i = \bigcap_{i\in I} \Log F_i\) \cite[Proposition~2.3]{BdRV}.

 Given a pair of frames~\(F=(X,\,R)\) and~\(G = (Y,\,S),\) a \emph{p-morphism from~\(F\) to~\(G\)} is a surjective map~\(f:\:X\to Y\) such that \(S\out(f(a)) = f(R\out(a))\) for any~\(a\in X.\)
 We write~\(F\pto G\) if there exists a p-morphism from~\(F\) to~\(G\). The p-morphism preserves the validity of modal formulas: if~\(F \pto G\), then~\(\Log F \subseteq \Log G\) \cite[Proposition~2.14]{BdRV}

 A \emph{frame isomorphism} between \(F=(X,\,R)\) and~\(G = (Y,\,S),\) is a bijection~\(f:\:X \to Y\) such that~\(aRb\) iff~\(f(a)Rf(b)\) for all~\(a,\,b\in X.\) It is straightforward that the existence of an isomorphism between~\(F\) and~\(G\) implies that~\(\Log F = \Log G.\)

\subsection{Classes of frames and their logics}\label{sec:conditions}

A relation \(R \subseteq X \times X\) is:
\begin{enumerate}[(i)]
   \item \emph{serial} if \(\forall a\in X\holds{R\out(a)\ne \varnothing};\)
   \item \emph{reflexive} if \(\forall a\in X\holds{a R a};\)
   \item \emph{irreflexive} if \(\lnot\exists a\in X\holds{a R a};\)
   \item \emph{symmetric} if \(\forall a,\,b\in X\holds{a R b \implies b R a};\)
   \item \emph{transitive} if \(\forall a,\,b,\,c \in X\holds{aRb,\,bRc \implies aRc};\)
   \item \emph{Euclidean} if \(\forall a,\,b,\,c \in X\holds{aRb,\,aRc \implies bRc};\)
   \item \emph{non-branching} if
   \(\forall a,\,b,\,c\in X\holds{aRb,\,aRc \implies bRc\text{ or }cRb\text{ or }c=b};\)
   \item \emph{Noetherian} if there are no infinite chains~\(a_0 R a_1 R \ldots\) with~\(a_i \ne a_{i+1},\,i\in \omega.\)
 \end{enumerate} 

A frame \(F = (X,\,R)\) is called serial (reflexive, etc.) if the relation of \(F\) is serial (reflexive, etc.)
  
In this paper we will consider the logics~\(\DV,\, \DIVV,\, \KVB,\, \SV,\) \(\GLIII,\,\GrzIII\). Recall that the frame classes of these logics are:
\begin{enumerate}
  \item \(\Fr\DV = \{\text{serial Euclidean frames}\};\)
  \item \(\Fr\DIVV = \{\text{serial transitive Euclidean frames}\};\)
  \item \(\Fr\KVB = \{\text{symmetric Euclidean frames}\};\)
  \item \(\Fr\SV = \{\text{reflexive Euclidean frames})\}\)
  \item \(\Fr\GLIII = \{\text{transitive irreflexive non-branching Noetherian frames}\}\)
  \item \(\Fr\GrzIII = \{\text{transitive reflexive non-branching Noetherian frames}\}\)
\end{enumerate}

These logics have the finite model property, so each of them is the logic of all finite point-generated frames that satisfy the corresponding frame condition \cite{BdRV}. For instance, \(\DV\) is the logic of all finite point-generated serial Euclidean frames.
\subsection{Random frames}

For any \(1\le n \in\omega,\) let~\([n]\) denote the set \(\{0,\,\ldots,\,n-1\}\), and let  \(\F_n = \{([n],\,R)\mid R \subseteq [n]\times [n]\}\) be the set of all frames with the set of states \([n].\) 

Let \(\C\) be a nonempty class of frames. For any \(1 \le n \in\omega,\) let \(\RF_n(\C)\) be the uniformly distributed random element of the finite set \(\F_n\cap \C:\)

\begin{equation}
  \P(\RF_n(\C)\in A) = \frac{|\F_n\cap \C\cap A|}{|\F_n\cap \C|}\quad\text{for any set }A\subseteq \F_n\cap \C.\label{eq:main}
\end{equation}
  
Formally, we fix some measure space~\((\Omega,\,\G,\,\P),\) where~\(\G\) is a \(\sigma\)-algebra on~\(\Omega\) and~\(\P:\:\G \to [0,1]\) is a measure, and define~\(\RF_n(\C)\) to be a measurable map from \(\Omega\) to~\(\F_n\) that satisfies~\eqref{eq:main}, where \(\P(\RF_n(\C)\in A)\) is a notation for ~\(\P\{\omega\in \Omega\mid \RF_n(\C)(\omega)\in A \}.\) The values~\(\RF_n(\C)(\omega)\) for~\(\omega\in \Omega\) are called \emph{realizations} of~\(\RF_n(\C).\)

For any set of frames \(\Q\subseteq \F_n\), we say that \(\RF_n(\C)\) belongs to \(\Q\) \emph{asymptotically almost surely (a.a.s.)} if
\[
  \lim_{n\to \infty}\P(\RF_n(\C)\in Q) = 1.
\]
Sometimes we will refer to a set of frames~\(\Q\subseteq \F_n\) as a \emph{property of frames}. In this case `\(\Q\) holds in~\(\RF_n(\C)\) a.a.s.' means that~\(\RF_n(\C)\in \Q\) a.a.s.

Let \(\Log\as(\C)\) denote the set of formulas \(\varphi\in \ML\) such that \(\RF_n(\C)\models\varphi\) asymptotically almost surely.

If \(L\) is a normal modal logic, we write \(\RF_n(L)\) for \(\RF_n(\Fr L)\) and \(L\as\) for~\(\Log\as(\Fr L).\)

By \eqref{eq:main},
\begin{equation}\label{eq:as-validity}
  \varphi\in \Log\as(\C) \iff \lim_{n\to \infty}\frac{|\F_n\cap \C\cap \Fr\{\varphi\}|}{|\F_n\cap \C|} = 1.
\end{equation}

The present work studies the sets \(\Log\as(\C)\) for some modally definable classes of frames: \(\C = \Fr(L)\) for some modal logic \(L.\)

\begin{theorem}\label{thm:basic}
  For any class of frames \(\C,\) \(\Log\as(\C)\) is a normal modal logic and \(\Log\as(\C) \supseteq \Log(\C).\)
\end{theorem}
\begin{proof}
  Since the minimal normal modal logic \(\K\) is valid in all frames, \(\P(\RF_n(\C)\models \K) = 1\) for any \(n \in\omega,\) hence \(\K \subseteq \Log\as(\C).\) Then \(\Log\as\) contains all propositional tautogolies and the normality axiom.

  Let us show that \(\Log\as(\C)\) is closed under \(\MP.\) Let \(\varphi\in \Log\as(\C)\) and \(\varphi\to \psi\in \Log\as(\C).\) Since the logic of any frame is a normal modal logic, \(\Fr\{\varphi,\,\varphi \to \psi\}\subseteq \Fr\{\psi\}.\) Then by \eqref{sec:preliminaries}
  \begin{align*}
    \P(\RF_n(\C)\models \psi) &= \frac{|\F_n\cap \C\cap \Fr\{\psi\}|}{|\F_n\cap \C|}\ge \frac{|\F_n\cap \C\cap \Fr\{\varphi,\,\varphi \to \psi\}|}{|\F_n\cap \C|} \\&= \P(\RF_n(\C)\models\varphi\text{ and }\RF_n(\C)\models\varphi\to\psi\}))\\& \ge 1 - \P(\RF_n(\C)\not\models \varphi) - \P(\RF_n(\C)\not\models\varphi\to\psi).
  \end{align*}
  Take the limit of both sides as \(n \to \infty.\) By assumption, \(\P(\RF_n(\C)\not\models \varphi) \to 0\) and \(\P(\RF_n(\C)\not\models\varphi\to\psi) \to 0,\) so \(\lim_{n\to \infty}\P(\RF_n(\C)\models \psi) \ge 1.\) Then \(\psi\in \Log\as(\C)\) by the definition.

  Since \(\Fr(\b\varphi) \subseteq\Fr(\varphi)\) for any \(\varphi \in \ML,\) \(\Log\as(\C)\) is closed under~\(\Gen\) by~\eqref{eq:as-validity}. A similar argument applies for \(\Sub.\)

  Finally, if \(\varphi\in \Log(\C)\), then \(\varphi\) is valid in any possible value of \(\RF_n(\C).\) Then for any \(n\in\omega,\,\P(\RF_n(\C)\models \varphi) = 1\), so \({\varphi\in \Log\as(\C)}.\)
\end{proof}

It follows directly from the theorem that \(L \subseteq L\as \) for any logic \(L.\)

  

\begin{proposition}\label{prop:squeeze}
  If~\(L\) is a modal logic, then~\((L\as)\as = L\as.\)
\end{proposition}
\begin{proof}
  By Theorem~\ref{thm:basic}~\(L\as \subseteq (L\as)\as.\) For the other direction, consider any \(\varphi\in (L\as)\as\).
  By the definition, there exist the limits:
  \begin{gather*}
    \lim_{n\to \infty}\frac{\left|\F_n\cap \Fr L\as \cap \Fr \{\varphi\}\right|}{\left|\F_n\cap \Fr L\as\right|} = \P(\RF_n(L\as)\models \varphi) = 1,\\
    \lim_{n\to \infty}\frac{\left|\F_n\cap \Fr L\as\right|}{\left|\F_n\cap \Fr L\right|} = \P(\RF_n(L) \models L\as) = 1.
  \end{gather*}
  Since~\(L \subseteq L\as\), we have \(\Fr L\as \subseteq \Fr L\) , so for any~\(\varphi\in (L\as)\as\),
  \begin{align*}
    \P(\RF_n(L)\models \varphi) &= \lim_{n\to \infty} \frac{\left|\F_n\cap \Fr L \cap \Fr \{\varphi\}\right|}{\left|\F_n\cap \Fr L\right|} 
    \\&\ge \lim_{n\to \infty} \frac{\left|\F_n\cap \Fr L\as \cap \Fr \{\varphi\}\right|}{\left|\F_n\cap \Fr L\right|}
    \\&= \lim_{n\to \infty}\frac{\left|\F_n\cap \Fr L\as \cap \Fr \{\varphi\}\right|}{\left|\F_n\cap \Fr L\as\right|} \cdot \frac{\left|\F_n\cap \Fr L\as\right|}{\left|\F_n\cap \Fr L\right|}
    \\&= \lim_{n\to \infty}\frac{\left|\F_n\cap \Fr L\as \cap \Fr \{\varphi\}\right|}{\left|\F_n\cap \Fr L\as\right|}\lim_{n\to \infty}\frac{\left|\F_n\cap \Fr L\as\right|}{\left|\F_n\cap \Fr L\right|}
    = 1 \cdot 1
    = 1.
  \end{align*}
\end{proof}

\subsection{Asymptotics}

Let \(f,\,g:\:\omega\to \omega\) be some functions. Then we write
\begin{enumerate}
  \item \(f\sim g,\) if \(\lim_{n\to \infty}\frac{f(n)}{g(n)}=1;\)
  \item \(f=o(g),\) if there exists~\(\alpha:\:\omega\to \mathbb{R}\) such that \(f(n) = \alpha(n) g(n)\) for all~\(n\in \omega\) and~\(\lim_{n\to \infty}\alpha(n) = 0;\)
  \item \(f=O(g),\) if there exists a real number \(C > 0\) and \(n\in\omega\) such that \(f(n) \le C g(n);\)
  \item \(f=\Omega(g),\) if there exists a real number \(C > 0\) and \(n\in\omega\) such that \(f(n) \ge C g(n).\)
\end{enumerate}

\subsection{Set partitions and combinatorial numbers}

Let \(X\) be a set. A family of subsets \(\U \subseteq \Pow{X}\setminus\{\varnothing\}\) is a \emph{partition}  of \(X\) if the elements of \(\U\) are pairwise disjoint and \(X = \bigcup \U.\)

The \emph{Bell number~\(B_n\)} is defined as the number of distinct partitions of the set \([n] = \{0,\,\ldots,\,n-1\}\). Equivalently, \(B_n\) is the number of distinct equivalence relations on~\([n].\) The growth rate of the Bell number is described by the asymptotic expression \cite[Section~6.2]{debrujin}
\begin{equation}
   \ln B_n = n(\ln n - \ln\ln n - 1 + o(1)),\quad n\to \infty.\label{eq:Bn}
 \end{equation}
 The Bell numbers satisfy
\begin{equation}\label{eq:Bn-quotient}
  \lim_{n\to \infty} \frac{B_n}{B_{n+1}}\to 0.
\end{equation}  
We give the proof of \eqref{eq:Bn-quotient} in Appendix.


Given \(r \in\omega,\) the number of partitions \(\U\) of \([n]\) such that \(|U|\le r\) for any \(U\in \U\) is denoted~\(G_{n,r}.\) The asymptotic and combinatorial behavior of~\(G_{n,r}\) is discussed in \cite{moser_wyman}.

In this paper we will use the following estimation, which we prove in Appendix. For any constant~\(r,\,k \in \omega,\)
\begin{equation}\label{eq:Gnr-and-Bn}
  G_{n,r} 2^{k n} = o(B_n),\,n\to \infty
\end{equation}

Given~\(n\in \omega\) and~\(m \le n,\) the \emph{binomial coefficient~\(\binom{n}{m}\)} is defined as the number of distinct~\(m\)-element subsets of the set~\([n].\) The following estimation holds for any~\(n\in \omega\) and~\(m \le n\)~\cite[Section 5.4]{asymptopia}:
\begin{equation}\label{eq:binomial}
  \binom{n}{m} \le \binom{n}{\lfloor n / 2 \rfloor} \sim \frac{2^n}{\sqrt{\pi n / 2}},\quad n\to \infty
\end{equation}

\section{Connected frames}
The frame classes of different modal logics can demonstrate a very intricate combinatorial behavior that complicates the direct computation of probabilities by~\eqref{eq:as-validity}. However, it turns out that under certain conditions finding \(\Log\as(\C)\) can be reduced to a much simpler problem of finding the almost sure validities in the connected frames of \(\C.\) 


Let \(F=(X,\,R)\) be a frame. Let \({\sim} = {(R\cup R\inv)^*}\). Then \(\sim\) is an equivalence relation on \(X.\) The elements of \(X / {\sim}\) are called the \emph{connected components} of \(X.\) If  \(X\) has exactly one connected component, \(F\) is called \emph{connected}.

For a class of frames \(\C,\) we denote \(\Con\C\) the class of all connected frames in \(\C.\)

\begin{proposition}\label{prop:connected-components}
  Let \(L\) be a modal logic such that \(\F_n\cap \Con\Fr L\) is nonempty for any~\(n\in \omega.\) Then
  \(|\F_n\cap \Fr L| \ge B_n\) for any~\(n<\omega.\)
\end{proposition}
\begin{proof}
  Let \(n \in \omega\) and let \(\U\) be a partition of~\([n].\) For any \(U\in\U\) there exists a frame \(F_U\in \F_{|U|}\cap \Con\Fr L.\) Construct a frame \(F_\U = ([n],R_\U)\) by putting a copy of \(F_U\) on \(U\) for any~\(U\in\U.\) Then \(F_\U\) is isomorphic to \(\biguplus_{U\in \U} F_U,\) and \(F_\U\models L\) since \(F_U \models L\) for any~\(U\in\U.\) Thus \(F_\U\in \F_n\cap \Fr L.\) 

  Observe that \(\U\) is exactly the set of connected components of \(F_\U.\) Thus if \(\U,\,\V\) are partitions of \(X\)  and \(F_\U = F_\V\), then \(\U = \V.\) Then \(\U \mapsto F_\U\) is an injective mapping of the partitions of \([n]\) into \(\F_n\cap \Fr L,\) so \(|\F_n\cap \Fr L| \ge B_n.\) 
\end{proof}
\begin{definition}
  Let \(r \in \omega.\) Denote by \(\F_n^{\le r}\) the set of frames in \(\F_n\) whose connected components have cardinality at most~\(r.\)
\end{definition}
\begin{proposition}\label{prop:small-components}
  For any~\(r\in \omega,\) \(\left|\F_n^{\le r}\right| = o(B_n),\,n \to \infty.\) 
\end{proposition}
\begin{proof}
  Let \(F=([n],R)\in \F_n^{\le r}.\) If \(\U\) is the set of connected components of \(F\), then \((a,b)\not\in R\) for any \(a\in U,\,b\in V\) where \(U,\,V\in \U,\,U\ne V.\) Then
  \begin{equation}\label{eq:R-partitioned}
    R \subseteq \bigcup_{U\in \U} U\times U.
  \end{equation}
  Then \(R \in \Pow{\bigcup_{U\in \U} U\times U}.\) By the assumption, \(|U|\le r\) for any \(U\in \U,\) so we can estimate 
  \begin{equation}\label{eq:count-relations}
    \left|\bigcup_{U\in \U} U\times U\right| \le \sum_{U\in \U} r^2 = |\U|r^2.
  \end{equation}

  Let \(A_\U \subseteq \F_n^{\le r}\) denote the set of all frames in \(\F_n\) whose set of connected components is~\(\U.\) By~\eqref{eq:R-partitioned} and~\eqref{eq:count-relations}, \(|A_\U| \le 2^{|\U| r^2}\). Therefore
  \[
    |\F_n^{\le r}| \le \sum_{\U} |A_\U|  \le \sum_\U 2^{r^2 |\U| } \le \sum_\U 2^{r^2 n},
  \]
  where the sum is taken over all partitions \(\U\) of \([n]\) into sets of cardinality at most~\(r.\) The number of such partitions is \(G_{n,r},\) so by \eqref{eq:Gnr-and-Bn}
  \[
    |\F_n|^{\le r} \le G_{n,r} 2^{r^2 n} = o(B_n).
  \]
\end{proof}
\begin{proposition}\label{prop:large-component-aas}
  Let \(L\) be a modal logic such that \(\F_n\cap \Con\Fr L\) is nonempty for any~\(n\in \omega.\) Then for any fixed~\(r \in \omega,\) \(\RF_n(L)\) has a connected component of cardinality greater than~\(r\) a.a.s.
\end{proposition}
\begin{proof}
  By Proposition~\ref{prop:connected-components} and Proposition~\ref{prop:small-components},
  \[
    \P(\RF_n(L)\in \F_n^{\le r}) = \frac{\left|\F_n^{\le r}\cap \Fr L\right|}{\left|\F_n\cap \Fr L\right|} \le \frac{o(B_n)}{B_n} \to 0,\,n\to \infty.
  \]
\end{proof}
To prove the following theorem, we consider the distributions of generated subframes of the random frame that have some fixed size~\(m\in \omega.\) To simplify the reasoning, we view them as random frames in~\(\F_m,\) using the following definition. 
\begin{definition}
  Let~\(m,\, n \in \omega,\,m \le n.\) If~\(U \subseteq [n]\) and~\(|U| = m,\) the \emph{monotone relabeling} of~\(U\) is the unique bijection~\(\alpha:\:U \to [m]\) such that~\(\alpha(a) \le \alpha(b)\) iff~\(a \le b.\) Two frames~\(F = (U, R)\) and~\(G= ([m], S)\) \emph{coincide up to monotone relabeling} if the monotone relabeling \(\alpha:\:U \to [m]\) is a frame isomorphism between~\(F\) and~\(G.\) 
\end{definition}
\begin{theorem}\label{thm:logic-of-connected}
  Let \(L\) be a modal logic such that:
  \begin{enumerate}
     \item \(\F_n\cap \Con\Fr L\) is nonempty for any~\(n\in \omega.\) 
     \item For any~\(\varphi\not\in \Log\as(\Con\Fr L),\)
     \begin{equation}\label{eq:limsup}
       \limsup_{n\to \infty} \P(\RF_n(\Con\Fr L)\models \varphi) < 1.
     \end{equation}
   \end{enumerate} Then
  \(L\as \subseteq \Log\as(\Con\Fr L).\)
\end{theorem}
\begin{proof}
  Suppose that \(\varphi\not\in\Log\as(\Con\Fr L),\) then by \eqref{eq:limsup} 
  \[
    \limsup_{n\to \infty}\P(\RF_n(\Con\Fr L)\models \varphi) = \limsup_{n\to \infty}\frac{|\F_n \cap \Con\Fr L \cap \Fr \{\varphi\}|}{|\F_n\cap \Con\Fr L|}<  1.
  \]
  Then for some \(r \in \omega\) and \(p > 0,\)
  \begin{equation}\label{eq:prob-phi-false}
    \frac{|\F_m \cap \Con\Fr L \setminus \Fr \{\varphi\}|}{|\F_m\cap \Con\Fr L|} > p \quad \forall m \ge r.
  \end{equation}

  Define a random subset \(\RU_n \subseteq [n]\) to be the connected component of \(\RF_n(L)\) that has the maximal cardinality. If there are more than one such components \(U_1,\,\ldots,\,U_k\), let \(\RU\) be the one that contains the state \(a = \min \bigcup_{j=1}^k U_k\).

  Any realization of the generated subframe \(\RF_n(L) \gen \RU_n\) is a connected frame that validates \(L\), so \(\RF_n(L) \gen \RU_n\) coincides, up to monotone relabeling of states, with some frame from \(\F_{|\RU_n|} \cap \Con\Fr L\) almost surely.

  We claim that for any fixed \(U \subseteq [n]\), the values of \(\RF_n(L) \gen \RU_n\) with \(\RU_n = U\) are distributed uniformly on \(\F_{|U|} \cap \Con\Fr L.\) Informally, \(\RF_n(L)\gen \RU_n\) is independent of \(\RF_n(L)\gen([n]\setminus \RU_n)\).

  Let~\(U \subseteq [n]\) and denote by~\(\G_{n,U} \) the set of frames from \( \F_n\cap \Fr L\) where~\(U\) is the maximal connected component. Let us consider any \(F_1,\,F_2 \in \F_{|U|} \cap \Con\Fr L\). For any~\(G=([n],R)\in \G_{n,U}\) such that~\(G \gen U\) equals~\(F_1\) (up to monotone relabeling, which we assume hereinafter), we construct a frame~\(G'\) by changing the relation of~\(G\) on \(U\) in such a way that~\(G' \gen ([n]\setminus U) = G \gen ([n]\setminus U)\) and \(G' \gen U = F_2.\)

  Let us show that~\(G' \in \G_{n,U}.\) By the construction,~\(U\) is the maximal connected component of~\(G'.\) Moreover, \(G\models L\) implies that~\(G\gen ([n]\setminus U) \models L,\) so \(G' \models L\) since~\(G' \cong G \gen ([n]\setminus U)\,\uplus\, F_2\) and~\(F_2\models L.\) Thus~\(G' \in \F_n\cap \Fr L.\)

  Then the mapping~\(G \mapsto G'\) is a bijection between \(\{F\in \G_{n,U}:\:F\gen U = F_1\}\) and~\(\{F\in \G_{n,U}:\:F\gen U = F_2\},\) so
  \[
    \left|\{F\in\G_{n,U}:\:F\gen U = F_1\}\right| = \left|\{F\in\G_{n,U}:\:F\gen U = F_2\}\right|;
  \]
  consequently, for any~\(U \subseteq [n]\) and \(F_1,\,F_2\in \F_{|U|}\cap \Con\Fr L,\)
  \begin{align}
    \P&\left(\RF_n(L) \gen \RU_n = F_1 \,\Big|\, \RU_n = U\right) 
      = \frac{\P\left(\RF_n(L) \gen \RU_n = F_1 \text{ and } \RU_n = U\right)}{\P(\RU_n = U)}\nonumber\\
      &= \frac{\left|\{F\in\G_{n,U}:\:F\gen U = F_1\}\right|}{|\G_{n,U}|}
      = \frac{\left|\{F\in\G_{n,U}:\:F\gen U = F_2\}\right|}{|\G_{n,U}|}\nonumber\\
      &= \frac{\P\left(\RF_n(L) \gen \RU_n = F_2 \text{ and } \RU_n = U\right)}{\P(\RU_n = U)} = \P\left(\RF_n(L) \gen \RU_n = F_2 \,\Big|\, \RU_n = U\right) \label{eq:uniform-distribution}
  \end{align}
  Since \(\RF_n(L) \gen \RU_n\) is a connected frame that validates \(L\),
  \[
    \sum_{G \in \F_{|U|}\cap \Con\Fr L} \P\left(\RF_n(L) \gen \RU_n = G \,\Big|\, \RU_n = U\right) = 1.
  \]
  By~\eqref{eq:uniform-distribution} all terms in this sum are equal. Thus for any \(G\in\F_{|U|}\cap \Con\Fr L,\)  
  \begin{equation}\label{eq:prob-of-G}
    \P\left(\RF_n(L) \gen \RU_n = G \,\Big|\, \RU_n = U\right) = \frac{1}{|\F_{|U|}\cap \Con\Fr L|}.
  \end{equation}
  Since \(\RF_n(L)\gen \RU_n\) is a generated subframe of \(\RF_n(L),\) \(\RF_n(L) \not\models \varphi\) whenever \(\RF_n(L)\gen \RU_n\not\models \varphi.\) Then by the law of total probability \cite[Proposition~4.1]{probability}
  \begingroup
  \allowdisplaybreaks
  \begin{align*}
    \P(\RF_n(L) \not\models \varphi) &\ge \P(\RF_n(L)\gen \RU_n \not\models \varphi) \\
    &= \sum_{U \subseteq [n]}\P\left(\RF_n(L)\gen \RU_n \not\models \varphi\,\Big|\, \RU_n = U\right) P(\RU_n = U)\\
    &\ge \sum_{|U| > r}\P\left(\RF_n(L)\gen \RU_n \not\models \varphi\,\Big|\, \RU_n = U\right)\P(\RU_n = U)\\
    &= \sum_{|U| > r}\ \sum_{\substack{G\in \F_{|U|} \cap \Con\Fr L \\G\not\models \varphi}}\P\left(\RF_n(L)\gen \RU_n = G\,\Big|\, \RU_n = U\right)\P(\RU_n = U)\\
    &\overset{\eqref{eq:prob-of-G}}{=} \sum_{|U| > r}\ \sum_{\substack{G\in \F_{|U|} \cap \Con\Fr L \\G\not\models \varphi}}  \frac{1}{|\F_{|U|}\cap \Con\Fr L|}\P(\RU_n = U)\\
    &= \sum_{|U| > r}\P(\RU_n = U) \sum_{\substack{G\in \F_{|U|} \cap \Con\Fr L \\G\not\models \varphi}}  \frac{1}{|\F_{|U|}\cap \Con\Fr L|}\\
    &= \sum_{|U| > r}  \P(\RU_n = U)\frac{|\F_{|U|} \cap \Con\Fr L \setminus \Fr \{\varphi\}|}{|\F_{|U|}\cap \Con\Fr L|}\\
    &= \sum_{m > r}\P(|\RU_n| = m)\frac{|\F_{m} \cap \Con\Fr L \setminus \Fr \{\varphi\}|}{|\F_{m}\cap \Con\Fr L|}\\
    &\overset{\eqref{eq:prob-phi-false}}{>}\sum_{m > r}\P(|\RU_n| = m) \cdot p = p \sum_{m > r} \P(|\RU_n| = m) 
    = p \P(|\RU_n| > r).
  \end{align*}
  \endgroup
  Take the limit as \(n\to \infty.\) By Proposition~\ref{prop:large-component-aas}, \(\P(|\RU_n| > r) \to 1,\) so
  \[
    \lim_{n\to \infty}\P(\RF_n(L) \not\models \varphi) \ge p \lim_{n\to \infty}\P(|\RU_n| > r) = p.
  \]
  Then \(\varphi\not\in L\as\) since 
  \[
    \lim_{n\to \infty}\P(\RF_n(L) \models \varphi) =  1 - \lim_{n\to \infty}\P(\RF_n(L) \not\models \varphi) \le 1 - p < 1.
  \]
\end{proof}
\section{Euclidean frames}
In this section we apply Theorem~\ref{thm:logic-of-connected} to study the almost sure validities in Euclidean frames.

We begin with a simple observation about the structure of a Euclidean frame.

\begin{proposition}\label{prop:U_F}
  Let \(F = (X,\,R)\) be Euclidean. Then there exists a unique subset \(U \subseteq X\) (possibly empty) such that \(R\restr U\) is an equivalence relation and \(R \subseteq X \times U.\)
\end{proposition}
\begin{proof}
  Let \(U = \bigcup_{a\in X} R\out(a).\) For any \(u\in U,\,a R u\) for some \(a\in X,\) so by the definition of Euclidean relation \(u R u.\) Then \(R \restr U\) is reflexive and Euclidean, hence an equivalence relation.

  Every state \(a\in X\) with \(R\IN(a) \ne \varnothing\) is in \(U\), so \(R \subseteq X \times U.\) 

  To prove the uniqueness, let us assume that \(V \subseteq X\) satisfies the conditions. If \(a \in V \setminus U\), then \(R\IN(a) = \varnothing,\) so \(R\) is not reflexive on \(V\setminus U.\) Since \(R\restr V\) is an equivalence relation, \(V \subseteq U.\) If \(a \in U\), then \(a R a\), so \((a,\,a) \in X \times V,\) hence \(a\in V.\) Then \(U = V\), so \(U\) is the unique subset of \(X\) that satisfies the conditions. 
\end{proof}
\begin{definition}
  For any Euclidean frame \(F,\) let \(U_F\) denote the subset of \(\dom F\) defined by the conditions from Proposition~\ref{prop:U_F}.
\end{definition}

\begin{definition}
  A frame \(F = (X,\,R)\) is a \emph{cluster} if \(R = X\times X.\)
\end{definition}
\begin{proposition}\label{prop:descr-con-K5}
  A serial Euclidean frame \(F\) is connected if and only if \(F\gen U_F\) is a cluster.
\end{proposition}
\begin{proof}
  Let \(F = (X,\,R)\) be connected. Since \(R\) is an equivalence relation on \(U_F,\) \((R^* \cup (R^*)\inv \cup Id_X)\restr U_F = R \restr U_F,\) so \(R\restr U_F = U_F\times U_F,\) thus \(F \gen U_F\) is a cluster. 

  Conversely, suppose that \(R \restr U_F\) is a cluster. For any \(a\in X,\) \(R\out(a) \subseteq U_F\) and \(R\out(a) \ne \varnothing,\) so \(aRu \) for some \(u\in U.\) Then \(a R^2 v\) for any \(v\in U.\) Therefore \(R^2\out[X] = U,\) so \(R^2 \cup R^{-2} = X\times X,\) and \(F\) is connected.
\end{proof}

\begin{proposition}\label{prop:count-con-K5}
  \(\left|\F_n\cap \Con\Fr \DV\right| = 2^{\frac{n^2}{4} + n + O(\log n)}.\)
\end{proposition}
\begin{proof}
  A frame \(F = ([n],\,R)\in \F_n\cap \Con\Fr\DV\) is uniquely determined by \(U_F\) and the family of subsets \(\{R\out(a)\mid a\in [n]\setminus U_F\}\) where \(R\out(a) \subseteq U_F\) and~\(R\out(a)\ne \varnothing\) for all \(a.\) Then 
  \begin{equation}\label{eq:count-con-K5}
    \left|\F_n\cap \Con\Fr \DV\right| = \sum_{\substack{U \subseteq [n]\\U\ne\varnothing}} \left(2^{|U|}-1\right)^{n-|U|} = \sum_{m=1}^n\binom{n}{m} \left(2^m -1\right)^{n-m}.
  \end{equation}
  We find the lower and the upper asymptotic bound for this sum using~\eqref{eq:binomial}:
  \[
    \sum_{m=1}^n \binom{n}{m} (2^m - 1)^{n-m} \le  \sum_{m=1}^n \binom{n}{m} 2^{\frac{n^2}{4}} = 2^{\frac{n^2}{4}}\sum_{m=1}^n \binom{n}{m} \le 2^{\frac{n^2}{4}} \cdot 2^n = 2^{\frac{n^2}{4} + n};
  \] \begin{multline*}
    \sum_{m=1}^n \binom{n}{m} (2^m-1)^{n-m} \ge \binom{n}{\frac{n}{2}} \left(2^{\frac{n}{2}}-1\right)^{n-\frac{n}{2}} \\= \frac{2^n}{\sqrt{\pi n / 2}}(1+o(1)) \cdot 2^{\frac{n^2}{4}}(1+o(1)) = 2^{\frac{n^2}{4} + n + O(\log n)}.
  \end{multline*}
\end{proof}

\begin{proposition}\label{prop:estimate-K5-1}
  For any fixed \(r \in\omega,\) \(|U_{\RF_n(\Con\Fr\DV)}| > r\) asymptotically almost surely.
\end{proposition}
\begin{proof}
  By \eqref{eq:count-con-K5}, the number of connected \(\DV\)-frames with \(|U_F|\le r\) is
  \[
    \sum_{U \subseteq [n],\, |U|\le r} 2^{|U|(n-|U|)} = \sum_{m=1}^r \binom{n}{m} 2^{m(n-m)} \le r \cdot O(n^r) \cdot 2^{nr} = O(2^{nr}),\,n\to \infty.
  \]
  Then by Proposition~\ref{prop:count-con-K5},
  \[
    \P(|U_{\RF_n(\Con\Fr\DV)}| \le r) = \frac{O(2^{nr})}{\left|\F_n\cap \Con\Fr \DV\right|} = \frac{O(2^{nr})}{2^{\frac{n^2}{4} + O(n)}} \to 0,\quad n\to \infty.
  \]
\end{proof}

\begin{proposition}\label{prop:estimate-K5-2}
  Let \(\RR_n\) denote the relation of  \(\RF_n(\Con\Fr\DV)\) and \(\RU_n = U_{\RF_n(\Con\Fr\DV)}.\) Let \(r \in\omega\) be fixed. Then \aas there exists a state \(a\in [n] \setminus \RU_n\) such that \(r < |\RR_n(a)| < |\RU_n|-r.\)
\end{proposition}
\begin{proof}
  Let \(\Q_{n,r}\) be the subset of \(F = ([n],\,R)\in\F_n\cap \Con\Fr\DV\) consisting of all frames~\(F = ([n],R)\) such that for any~\(a\in [n] \setminus U_F,\) \({|R\out(a)| \le r}\) or \(|R\out(a)| \ge |U_F|-r\). Let us estimate \(|\Q_{n,r}|.\) If we fix \(U_F \subseteq [n]\) with \(|U_F| = m\), then for any \(a\not\in U_F\) the nonempty subset \(R\out(a) \subseteq U_F\) can be chosen in \(\sum_{k=1}^r \left(\binom{m}{k} + \binom{m}{m-k}-1\right)\) ways. Note that
  \[
    \sum_{k=1}^r \left(\binom{m}{k} + \binom{m}{m-k}-1\right) \le 2r \binom{m}{r} \le 2r \frac{m^r}{r!} \le 2m^r.
  \]
  Then by \eqref{eq:count-con-K5},
  \begin{multline*}
    |\Q_{n,r}| \le \sum_{m=1}^n \binom{n}{m} \left(2m^r\right)^{n-m} \le \sum_{m=1}^n \binom{n}{m} (2m)^{nr} \\\le (2n)^{nr}\sum_{m=1}^n \binom{n}{m} \le (2n)^{nr} 2^{n} = O\left(2^{rn\log_2 n}\right),\quad n\to \infty.
  \end{multline*}
  Finally, we show that \(\RF_n(\Con \Fr \DV)\not\in \Q_{n,r}\) \aas By Proposition~\ref{prop:count-con-K5},
  \[
    \P(U_{\RF_n(\Con\Fr\DV)}\in \Q_{n,r}) = \frac{|\Q_{n,r}|}{\left|\F_n\cap \Con\Fr \DV\right|} \le \frac{O(2^{rn\log_2 n})}{2^{\frac{n^2}{4} + O(n)}} \to 0,\,n\to \infty.
  \]
\end{proof}
  
\begin{proposition}\label{prop:log-con-K5}
  For any \(\varphi\not\in \DV,\,\RF_n(\Con\Fr \DV)\not\models \varphi\) \aas
\end{proposition}
\begin{proof}
  Consider a formula \(\varphi\not\in \DV.\) There exists a finite point-generated \(\DV\)-frame \(G = (Y,S) = G\gen c\) such that \(G\not\models \varphi.\) Let \(r = |\dom G|.\)  Let \(F=(X,\,R)\) be a connected \(\DV\)-frame such that
  \begin{equation}\label{eq:properties-for-p-morphism}
    |U_F| > r\text{ and }\exists a\in X\setminus U_F\holds{r < |R\out(a)| < |U_F|-r}.
  \end{equation}
  We define a p-morphism \({f:F\gen a \pto G}.\) By Proposition~\ref{prop:descr-con-K5}, \(\dom F\gen a = R^*\out(a) = U_F \cup\{a\}.\) Let \(f(a) = c.\) Observe  that \(|R\out(a)| > r \ge |U_G| > |S(c)|,\) so let \(f\) map \(R\out(a)\) surjectively onto \(S(c).\) Analogously, let \(f\restr_{U_F\setminus R\out(a)}\) be a surjection onto \(U_G \setminus S(c).\)

  By Propositions~\ref{prop:estimate-K5-1} and \ref{prop:estimate-K5-2}, \(\RF_n(\Con\Fr \DV)\) has the property~\eqref{eq:properties-for-p-morphism} a.a.s., so \(\RF_n (\Con\Fr \DV)\pto G\), hence \(\RF_n(\Con\Fr \DV)\not\models \varphi\) \aas
\end{proof}
\begin{theorem}\label{thm:K5}
  \(\DV\as = \DV.\)
\end{theorem}
\begin{proof}
  By Theorem~\ref{thm:basic}, \(\DV \subseteq \DV\as\). For the other direction, observe that the cluster~\(([n],\,[n]\times[n])\) is a connected \(\DV\)-frame for all~\(n\in\omega\), thus~\(\F_n\cap \Con\Fr \DV\ne \varnothing\) for all~\(n\in\omega.\) By Proposition~\ref{prop:log-con-K5},~\(\limsup\limits_{n\to \infty}\P(\RF_n(\Con\Fr \DV)\models \varphi) = 0 < 1.\) Then by Theorem~\ref{thm:logic-of-connected}, \(\DV\as \subseteq \Log\as(\Con\Fr\DV) \subseteq \DV,\) where the latter inclusion also follows from Proposition~\ref{prop:log-con-K5}.
\end{proof}
Now we use our method to find~\(\DIVV\as,\, \KVB\as,\, \SV\as.\) We follow mostly the same strategy, so we omit some technical steps.
\begin{theorem}
  \(\DIVV\as = \DIVV.\)
\end{theorem} 
\begin{proof}
  Let~\(\varphi\not\in \DIVV,\) then \(G\not\models \varphi\) for some finite point-generated frame \(G\models \DIVV\).
  Let us show that~\(\RF_n(\Con\Fr \DIVV)\pto G\) a.a.s. Notice that a Euclidean frame \(F=(X,R)\) is serial, transitive and connected iff~\(U_F\ne \varnothing\) and~\(R = X \times U_F\). Then such frame is uniquely determined by its cluster~\(U_F\ne \varnothing\). Therefore 
  \[
    |\F_n\cap \Con\Fr\DIVV| = \left|\{U \subseteq [n]:\:U \ne \varnothing\}\right| = 2^n - 1.
  \]
  For any fixed~\(r \in \omega\) we have \(r < |U_{\RF_n(\Con\Fr \DIVV)}| < n\) a.a.s. Indeed, for any~\(m\le n\) there are~\(\binom{n}{m}\) choices for a cluster of size~\(m\) in~\(n,\) so
  \begin{align*}
    &\P(|U_{\RF_n(\Con\Fr \DIVV)}| \le r)
    = \frac{\sum_{m = 1}^r \binom{n}{m}}{2^n - 1}
    = \frac{O(n^r)}{2^n - 1} \to 0,&n\to \infty,\\
    &\P(|U_{\RF_n(\Con\Fr \DIVV)}| = n) = \frac{1}{2^n - 1} \to 0,&n\to \infty.
  \end{align*}
  There a.a.s. exists a p-morphism from some generated subframe of~\(\RF_n(\Con\Fr\DIVV)\) to~\(G,\) so \(\P(\RF_n(\Con\Fr \DIVV)\models \varphi)\to 0\) as~\(n\to \infty.\)

  Since~\(\varphi\not\in \DIVV\) was arbitrary,~\(\Log\as(\Con\Fr\DIVV) \subseteq \DIVV.\)

  By Theorem~\ref{thm:logic-of-connected} we have \(\DIVV\as \subseteq \Log\as(\Con\Fr\DIVV) \subseteq \DIVV.\) The converse inclusion follows from Theorem~\ref{thm:basic}.
\end{proof}
\begin{theorem}
  \(\KVB\as = \KVB.\)
\end{theorem}
\begin{proof}
  Notice that~\(\KVB\) is the logic of its finite point-generated frames, which are exactly the finite clusters and the irreflexive singletons. If~\(\varphi\not\in \KVB\) for some formula~\(\varphi,\) then either~\(([r],\,[r]\times[r]) \not\models \varphi\) for some~\(r\in \omega\), or~\((\{a\}, \varnothing)\not\models \varphi.\)

  Let~\(r\in \omega\) and~\(([r],\,[r]\times[r]) \not\models \varphi\).
  The connected components of a~\(\KVB\)-frame are clusters and irreflexive singletons.
  Then \(\RF_n(\KVB)\) a.a.s. contains a cluster of size greater than~\(r\) by Proposition~\ref{prop:large-component-aas}, so there is a p-morphism from a generated subframe of~\(\RF_n(\Con\Fr\KVB))\) to \(([r],[r]\times[r])\), so~\(\P(\RF_n(\Con\Fr\KVB)\models \varphi)\to 0\) as~\(n\to \infty.\)

  Next we consider a formula~\(\varphi\) such that \((\{a\},\varnothing\not\models \varphi)\).
  A \(\KVB\)-frame with set of states~\(n\) is uniquely determined by a subset~\(E\subseteq [n]\) that consists of the irreflexive singletons and an equivalence relation on \(n\setminus E.\) Then
  \[
    \left|\F_n\cap \KVB\right| = \sum_{E\subseteq [n]} B_{n - |E|} = \sum_{m=0}^n\binom{n}{m} B_{n-m} = \sum_{m=0}^n \binom{n}{m} B_n = B_{n+1}.
  \]

  A \(\KVB\)-frame has no irreflexive singletons iff its relation is an equivalence relation, so there are exactly~\(B_n\) frames without irreflexive singletons in~\({\F_n\cap \Fr\KVB}.\) Then by \eqref{eq:Bn-quotient}
  \[
    \P(\RF_n(\KVB)\text{ has no irreflexive singletons}) = \frac{B_n}{\left|\F_n\cap \KVB\right|} = \frac{B_n}{B_{n+1}}\to 0
  \]
  as \(n\to \infty.\) Then~\(\RF_n(\Con\Fr\KVB)\) contains a generated subframe isomorphic to~\((\{a\},\varnothing)\) a.a.s., so~\(\P(\RF_n(\Con\Fr\KVB)\models \varphi)\to 0\) as~\(n\to \infty.\)

  By Theorem~2.6,~\(\KVB\as \subseteq \Log\as(\Con\Fr\KVB) \subseteq \KVB.\) The converse inclusion is true by Theorem~\ref{thm:basic}.
\end{proof}
\begin{theorem}
  \(\SV\as = \SV.\)
\end{theorem}
\begin{proof}
  Observe that~\(\F_n\cap \Con\Fr\SV\) consists of one frame~\((n,\,n\times n).\) The finite point-generated frames of~\(\SV\) are the finite clusters, which are p-morphic images of~\((n,\,n\times n)\) for~\(n\) sufficiently large. Then \(\P(\RF_n(\Con\Fr\SV)\models \varphi) \to 0\) for any \(\varphi\not\in\SV.\) Then the statement of this theorem follows from Theorem~\ref{thm:basic} and Theorem~\ref{thm:logic-of-connected}.
\end{proof}
\section{Transitive frames}
In this section we discuss the logics of almost sure validities in the finite frames of~\(\GLIII\) and~\(\GrzIII.\)

\begin{definition}
  A frame~\(F = (X,\,R)\) is an \emph{inverse tree} if there is a unique element~\(a_0\) such that~\(R\out(a_0) = \varnothing\), and for any~\(a\in X\setminus \{a_0\},\) \(|R\out(a)| = 1 \) and~\(a R^* a_0.\) We denote \(T_n\) the set of all inverse trees over the set of states~\([n].\)
\end{definition}
\begin{definition}
  The \emph{distance} between states~\(a\) and~\(b\) of a frame~\(F = (X,\,R)\) is the number
  \(d_R(a,\,b) = \min\{n \in \omega:\:a R^n b\}\) for any \( a,\,b\in [n].\)
\end{definition}
\begin{definition}
  The \emph{height of a finite inverse tree} is the maximum of the distance between the states in that tree.
\end{definition}

\begin{definition}
  Given a transitive Noetherian relation~\(R\) on a set~\(X,\) the \emph{transitive reduction} of~\(R\) is defined as~\(R^- = R \setminus R^2.\) 
\end{definition}
It is straightforward to see that for a transitive Noetherian relation~\(R,\) its transitive reduction \(R^-\) is the smallest relation on~\(X\) such that~\({R \subseteq (R^-)^*}\) \cite{transitive_reduction}. If~\((X,\,R)\in \Con\Fr\GLIII\), then~\(R = (R^-)^+\). Similarly, if~\((X,\,R)\) is a frame in \(\Con\Fr\GrzIII\), then~\(R = (R^-)^*.\) Therefore the mapping~\(R \mapsto R^-\) has an inverse on \(\Con\Fr\GLIII\) and on~\(\Con\Fr\GrzIII.\)
\begin{proposition}\label{prop:tree-bijection}
  For any~\(n\in \omega,\) the sets of frames \(T_n,\) \(\F_n\cap \Con\Fr\GLIII,\) and  \(\F_n\cap \Con\Fr\GrzIII,\) are bijective.
  Moreover, the bijection preserves the distance between any pair of states.
\end{proposition}
\begin{proof}
  Follows directly from the discussion above.
\end{proof}

\begin{proposition}\label{prop:tree-height}
  Consider the random inverse tree with uniform distribution on \(T_n\) and let~\(\Rh_n\) be its height.

  There exists a constant~\(q>0\) such that for any fixed~\({r \in \omega},\) 
  \[
    \lim_{n\to \infty}\P(\Rh_n > r) \ge q.
  \]
\end{proposition}
\begin{proof}
  The asymptotic expressions for the expected value and the variance of~\(h_n\) are given in \cite{height}:
  \[
    \E (\Rh_n) \sim \sqrt{2\pi n};\quad \var (\Rh_n) \sim \frac{\pi(\pi-3)n}{3},\quad n\to \infty.
  \]
  By Chebyshev's inequality \cite[Theorem~1.4]{probability},
  \begin{multline*} 
     \P(\Rh_n\le r) \le \P\left(|\Rh_n-\E \Rh_n| \ge \E \Rh_n - r\right) \le \frac{\var (\Rh_n)}{(\E \Rh_n - r)^2}\\\sim \frac{\pi(\pi-3)n}{3(\sqrt{2\pi n}-r)^2}
     \sim \frac{\pi(\pi-3)}{3(\sqrt{2\pi}-\frac{r}{\sqrt{n}})^2}\sim \frac{\pi-3}{6}<1,\quad n\to \infty.
  \end{multline*}
  Let \(q = 1 - \frac{\pi-3}{6} > 0,\) then
  \(\lim_{n\to \infty}\P(\Rh_n> r) = 1-\lim_{n\to \infty}\P(\Rh_n\le r) \ge q.\)
\end{proof}
\begin{theorem}
  \(\GLIII\as = \GLIII.\)
\end{theorem}
\begin{proof}
  Let us recall that the logic~\(\GLIII\) is complete w.r.t. the finite irreflexive chains. Let~\({\varphi\not\in \GLIII}\). Then~\(\varphi\) is falsified in some irreflexive chain~\(F\) of a finite cardinality~\(r.\)

  Let~\(\RR_n\) denote the relation of~\(\RF_n(\Con\Fr\GLIII).\) By Proposition~\ref{prop:tree-bijection} and Proposition~\ref{prop:tree-height},
  \[
    \lim_{n\to \infty}\P\left(\exists a\in [n]:\:|\RR_n(a)| \ge r\right) = \lim_{n\to \infty}\P(\Rh_n > r) \ge q.
  \]
  With an asymptotic probability~\(q > 0\) there exists~\(a\in [n]\) such that~\({|\RR_n(a)|\ge r}.\) In this case \(\RF_n(\GLIII)\gen a\) is an irreflexive chain of cardinality at least~\(r\), so~\(F\) is isomorphic to some generated subframe of~\(\RF_n(\GLIII)\gen a.\)

  Then~\(\limsup_{n\to \infty}\P({\RF_n(\Con\Fr\GLIII)\models \varphi}) \le q < 1.\)

  Since~\(\varphi\not\in \GLIII\) was arbitrary,~\(\Log\as(\Con\Fr\GLIII) = \GLIII.\)

  Since \(([n],<)\) is a connected~\(\GLIII\)-frame,~\(\F_n\cap \Con\Fr \GLIII\ne \varnothing\) for all~\(n\in\omega\), thus by Theorem~\ref{thm:logic-of-connected}~\(\GLIII\as \subseteq \GLIII\). The converse is true by Theorem~\ref{thm:basic}.
\end{proof}
\begin{theorem}
  \(\GrzIII\as = \GrzIII.\) 
\end{theorem}
\begin{proof}
  Analogous to the previous theorem.
\end{proof}

\section{Results and discussion}
We developed several general results about the almost sure validities in random Kripke frames. Theorem \ref{thm:basic} states that the almost sure validities in the random frame~\(\RF(\C)\) is a normal modal logic that extends~\(\C\) for any class of frames~\(\C\). Theorem~\ref{thm:logic-of-connected}, which applies for many commonly studied modal logics, establishes an important inclusion~\(L\as \subseteq \Log\as(\Con\Fr L)\).

We established the axiomatizations for the considered classes of frames:
\begin{align*}
  &\DV\as = \DV;&
  &\DIVV\as = \DIVV;&
  &\KVB\as = \KVB;\\
  &\SV\as = \SV;&
  &\GrzIII\as = \GrzIII;&
  &\GLIII\as = \GLIII.
\end{align*}
Interestingly, all these logics share some desirable properties, such as finite axiomatization, finite model property, decidability, etc. This stands in contrast to the known results on~\(\K\as\) \cite{goranko} and~\(\textbf{GL}\as\) \cite{verbrugge} that imply that these logics lack the finite axiomatizability.

Our results on~\(\DV,\,\DIVV,\,\KVB,\,\SV,\,\GrzIII,\,\GLIII\) provide examples of logics that are equal to their `almost sure' counterparts. Finding a criterion that characterizes the logics with this property is an interesting direction for future research.

The computational method we use in this paper seems to be able to yield more general results, such as a classification of logics of almost sure validities of the frame classes of all logics above~\(\KV.\) We also conjecture that many of such logics obey the zero-one law.
\bibliographystyle{amsalpha}
\bibliography{references}
\section{Appendix}
\subsection{Proof of \eqref{eq:Bn-quotient}}
The Bell numbers satisfy Dobi\'{n}ski's formula \cite{dobinski}
\begin{equation}\label{eq:dobinski}
  B_n = e\inv\sum_{k=0}^\infty \frac{k^n}{k!},\quad k\in \omega.
\end{equation}
Let \(\RX\) be a random variable of with the standard Poisson distribution:
\[
  \P(\RX = k) = \frac{e\inv}{k!}.
\]
Then by \eqref{eq:dobinski} the~\(n^{th}\) moment of~\(\RX\) is~\(B_n:\)
\[
  \E(\RX^n) = \sum_{k=0}^\infty k^n \frac{e\inv}{k!} = B_n
\]
Since the function~\(\varphi:\:[0,\infty) \to [0,\infty),\,\varphi(x) = x^{\frac{n+1}{n}}\) is convex, Jensen's inequality \cite[Theorem~5.1]{probability} holds: \(\varphi(\E\RX^n)\le \E(\varphi(\RX^n))\), so
\[
  {B_n}^{\frac{n+1}{n}} \le \E\left(\RX^{\frac{n+1}{n}}\right) = \sum_{k=0}^\infty \left(k^n\right)^{\frac{n+1}{n}} \frac{e\inv}{k!} = e\inv \sum_{k=0}^n \frac{k^{n+1}}{k!} = B_{n+1},
\]
therefore
\[
  \frac{B_n}{B_{n+1}} \le \frac{B_n}{B_n^{\frac{n+1}{n}}} = B_n^{-\frac{1}{n}}.
\]
Recall that by \eqref{eq:Bn}~\(B_n = e^{n\ln n(1+o(1))}\), so~\(B_n^{-\frac{1}{n}} = e^{-\ln n(1+o(1))} \to 0\) as~\(n\to \infty.\) We conclude that
\[
  \lim_{n\to \infty}\frac{B_n}{B_{n+1}} = 0.
\]
\subsection{Proof of \eqref{eq:Gnr-and-Bn}}
An asymptotic expression of~\(G_{n,r}\) for fixed \(r\in \omega\) and \(n\to \infty\) is provided in~\cite{moser_wyman}:
\begin{equation}\label{eq:Gnr}
  G_{n,r}\sim\left(\frac{n}{R_{n,r}}\right)^n r^{-\frac{1}{2}}\exp\left(\frac{n}{R_{n,r}}+\frac{R_{n,r}{}^r}{r!}-n-1\right),
\end{equation}
where~\(R_{n,\,r}\) is the positive root of the equation
\begin{equation}\label{eq:R}
   R_{n,r}+\frac{R_{n,r}{}^2}{1!}+\frac{R_{n,r}{}^3}{2!}+\ldots+\frac{R_{n,r}{}^r}{(r-1)!}=n.
\end{equation}
Note that~\(R_{n,r} \to \infty\) as~\(n\to \infty.\) Then \({R_{n,r}}^k = o(R_{n,r}^r)\) for any~\(k < r,\) so the equation \eqref{eq:R} yields
\[
  \frac{R_{n,r}{}^r}{(r-1)!}(1+o(1)) = n,
\]
therefore
\[
  R_{n,r} = \left(n(r-1)!\right)^\frac{1}{r}(1+o(1).
\]

Then we may estimate~\(\ln R_{n,r} = \frac{1}{r} \ln n (1+o(1));\;\frac{n}{R_{n,r}}+\frac{R_{n,r}{}^r}{r!}-n-1 = O(n),\) so \eqref{eq:Gnr} implies
\[
  \ln G_{n,r} = n\ln n - \frac{1}{r}n\ln n + O(n).
\]
Therefore for any~\(k \in \omega\), by the asymptotic expression \eqref{eq:Bn} for~\(B_n\) we get:
\begin{multline*}
  \ln \left(\frac{G_{n,r} 2^{kn}}{B_n}\right) = n\ln n - \frac{1}{r}n\ln n + O(n) + kn\ln 2 - n\ln n(1+o(1)) \\= -\frac{1}{r}n\ln n(1+o(1)) \to -\infty;
\end{multline*}
then by exponentiating we get the desired estimation:
\[
  \frac{G_{n,r} 2^{kn}}{B_n} \to 0,\quad n\to \infty.
\]
\end{document}